\documentclass{amsart}
\usepackage{graphicx}
\usepackage{amsfonts}
\newtheorem{tm}{Theorem}

\newtheorem{defi}{Definition}

\newtheorem{rem}{Remark}
\newtheorem{rems}{Remarks}
\newtheorem{lm}{Lemma}
\newtheorem{ex}{Example}
\newtheorem{cor}{Corollary}
\newtheorem{prop}{Proposition}
\newtheorem{nota}{Notation}

\begin{document}

\title{Hyperbolic polynomials and canonical sign patterns}
\author{Vladimir Petrov Kostov}
\address{Universit\'e C\^ote d’Azur, CNRS, LJAD, France} 
\email{vladimir.kostov@unice.fr}

\begin{abstract}
  A real univariate polynomial is hyperbolic if all its roots are 
  real. By Descartes' rule of signs a hyperbolic polynomial
  (HP) with all coefficients nonvanishing
  has exactly $c$ positive and exactly $p$ negative roots counted with
  multiplicity, where $c$ and $p$ are the numbers of sign changes and sign
  preservations in the sequence of its coefficients. We
  discuss the question: If the moduli of all $c+p$ roots are
  distinct and ordered on the positive half-axis, then at which positions
  can the $p$ moduli of negative roots be depending on the positions of the
  positive and negative signs of the coefficients of the polynomial?
  We are especially interested in the choices of 
  these signs for which exactly one order of the moduli
  of the roots is possible.\\ 

  {\bf Key words:} real polynomial in one variable; hyperbolic polynomial; sign
  pattern; Descartes' 
rule of signs\\ 

{\bf AMS classification:} 26C10; 30C15
\end{abstract}
\maketitle

\section{Introduction}

We consider real univariate polynomials with nonvanishing coefficients. Such a
polynomial is {\em hyperbolic} if all its roots are real. Various problems
concerning hyperbolic polynomials (HPs) are exposed in~\cite{Ko}.
In this paper we
discuss the following question: Suppose that the moduli of all roots of a HP
are distinct and ordered on the positive half-axis. Then at which positions
can the moduli of the negative roots be depending on the signs of the
coefficients of the HP? In this sense we say that we are interested in the
possible orders on the positive half-axis of the moduli of roots of
HPs with given signs of their coefficients. 

Without loss of generality we consider only monic polynomials. 
A {\em sign pattern (SP)} is a finite sequence of $(+)$- and/or $(-)$-signs.
The SP defined by the polynomial $P:=\sum _{j=0}^da_jx^j$,
$a_j\in \mathbb{R}^*$, $a_d=1$, is the vector

$$\sigma (P)~:=~(~+~,~{\rm sgn}(a_{d-1})~,~{\rm sgn}(a_{d-2})~,~\ldots ~,~
{\rm sgn}(a_0)~)~.$$

\begin{nota}\label{nota1}
{\rm When we write
  $\sigma (P)=\Sigma _{m_1,m_2,\ldots ,m_s}$, $m_i\in \mathbb{N}^*$,
  $m_1+\cdots +m_s=d+1$,
this means that the SP $\sigma (P)$ begins with 
a sequence of $m_1$ signs $+$ followed by a sequence of $m_2$ signs
$-$ followed by a sequence of $m_3$ signs $+$ etc. The number $s-1$ is the
number of sign changes and the number $d-s+1$ is the number of sign
preservations of the SP $\sigma (P)$.
}
\end{nota}

The classical Descartes' rule of signs
says that the polynomial $P$ has not more than $s-1$ positive roots.
When applied
to the polynomial $P(-x)$, this rule implies that $P$ has not more than $d-s+1$
negative roots. Hence if $P$ is hyperbolic, then it has exactly $s-1$ positive
and exactly $d-s+1$ negative roots (all roots are counted with multiplicity).

\begin{rem}
  {\rm Fourier has made Descartes' rule of signs about real
    (but not necessarily hyperbolic) polynomials more precise by showing that
    the number of positive roots differs from $s-1$ by an even integer,
    see~\cite{Four}.
    For such polynomials, Descartes' rule
    of signs proposes only necessary conditions. Attempts to clarify the
    question how far from sufficient they are have been carried out in
    \cite{AlFu}, \cite{AJS}, \cite{FoKoSh}, \cite{FoNoSh}, \cite{Gr},
         \cite{KoCzMJ} and~\cite{KoMB}.}
  \end{rem}

\begin{defi}\label{defiCPP}
  {\rm Given a SP (of length $d+1$ and beginning with $+$) 
    we construct its corresponding
    {\em change-preservation pattern (CPP)} (of length~$d$) as follows.
    For $j\geq 2$, there is a $p$ (resp. a $c$)
    in position $j-1$ of the CPP if in positions $j-1$ and $j$ of the SP there
    are two equal (resp. two different) signs. It is clear that the
    correspondence between SPs beginning with $+$ and CPPs is bijective.
    Example: for $d=6$,
    to the SP $\sigma _0:=(+,+,-,-,+,+,+)$ there corresponds the CPP
    $(p,c,p,c,p,p)$.}
\end{defi}

\begin{defi}\label{defiCO}
  {\rm (1) Suppose that a degree $d$ HP $P$
    is given which defines the SP $\sigma$ of length $d+1$,
    suppose that the moduli of its roots are ordered on the real positive
    half-line, and suppose that all moduli of roots are distinct. We define
    formally the {\em canonical order} of the moduli of roots like this:
    the CPP corresponding to the given SP $\sigma$ is read from the back,
    each $p$ is replaced by an $N$ and
    each $c$ by a $P$. For the SP
    $\sigma _0$ from Definition~\ref{defiCPP} this gives $(N,N,P,N,P,N)$ which
    means that the moduli of the roots are $0<\gamma _1<\cdots <\gamma _6$,
    where the polynomial has positive roots $\gamma _3$ and $\gamma _5$, and
    negative roots $-\gamma _1$, $-\gamma _2$, $-\gamma _4$ and~$-\gamma _6$.
    \vspace{1mm}
    
    (2) For a HP $P$ and the SP $\sigma (P)$,
    we say that the SP $\sigma (P)$ is {\em realizable} by~$P$.}
\end{defi}

\begin{prop}\label{propevery}
  Every SP $\sigma$ of length $d+1$, $d\geq 1$, is realizable by a degree $d$
  HP with canonical order of the moduli of its roots.
\end{prop}

\begin{proof}
  We construct the HP in $d$ steps. At the first step we set $P_1:=x+1$ if the
  first component of the CPP is a $p$ and $P_1:=x-1$ if it is a $c$. Suppose
  that the degree $k$ HP $P_k$ is constructed which defines the SP $\sigma _k$
  obtained from $\sigma$ by deleting its last $d-k$ components. Set
  $P_{k+1}(x):=P_k(x)(x+\varepsilon )$ if the last two components of
  $\sigma _{k+1}$ are equal or $P_{k+1}(x):=P_k(x)(x-\varepsilon )$ if they are
  different, where $\varepsilon >0$. One chooses $\varepsilon$ so small that: 
 \vspace{1mm}
 
 1) the signs of the first $k+1$ coefficients of $P_{k+1}$ are the same as the
 ones of $P_k$;
 \vspace{1mm}
 
 2) the number $\varepsilon$ is smaller than all the moduli of roots of $P_k$.
 \vspace{1mm}
 
 It is clear that for $k=d$, the HP $P_d$ thus obtained defines the SP
 $\sigma$ and that the order of the moduli of its roots is the canonical one.
\end{proof}

\begin{rems}\label{rems1}
  {\rm (1) The proposition can be generalized for real, but not necessarily
    hyperbolic polynomials, see Lemmas~14 and 17 in~\cite{FoKoSh}. The way of
    constructing new polynomials by adding new roots of modulus much smaller
    than the already existing moduli (which preserves the signs of the first
    $d+1$ coefficients) can be called {\em concatenation of polynomials
      (or of SPs)}.
    The construction described in the proof of Proposition~\ref{propevery}
    extends at each step the SP by adding a $(+)$- or $(-)$-sign at its rear.
    \vspace{1mm}
    
    (2) One
    can propose a similar concatenation, i.e.
    construction of HPs, in which each new root has
    a modulus much larger than the moduli of the already existing roots.
    Namely, given a degree $d$ HP $P(x)$ with no vanishing coefficients
    one considers the HP $(1\pm \varepsilon x)P(x)$
    which for $\varepsilon >0$ sufficiently small has the same signs
    of the last $d+1$ coefficients as $P$. Its new root equals
    $1/(\mp \varepsilon )$. After this one has to multiply the polynomial
    by $\pm 1/\varepsilon$ to make it monic again. This construction extends
    the SP by adding a $(+)$- or $(-)$-sign at its front.}
  \end{rems}

\begin{defi}
  {\rm A SP is called {\em canonical} if it is realizable only by HPs with
    canonical order of the moduli of their roots.}
\end{defi}

\begin{ex}\label{ex1}
  {\rm (1) The following SPs $\Sigma _{m_1,m_2,\ldots ,m_s}$ are canonical:}

    $$\Sigma _{m_1,1}~,~\Sigma _{1,m_2}~,~\Sigma _{m_1,1,m_3}~~~{\rm and~for~}m_2\geq 3,~~~
    \Sigma _{1,m_2,1}~,$$
           {\rm see Theorem~1, Corollary~1, Theorem~5 and Theorem~2 in
             \cite{Koarxiv} respectively. The SP $\Sigma _{1,2,1}$ is not
             canonical -- by part (1) of Example~2 therein
             the SP $\Sigma _{1,2,1}$ is realizable by each of the three
             polynomials $(x+1)(x-1.5)(x-1.6)$, $(x+1)(x-1.5)(x-0.6)$ and
             $(x+1)(x-0.5)(x-0.6)$.
             \vspace{1mm}
             
             (2) For $m_1\geq 2$, $m_2\geq 2$, the SP $\Sigma _{m_1,m_2}$ is not
             canonical, see Theorem~1 and Corollary~1 in~\cite{Koarxiv}.}
\end{ex}

In the present paper we give sufficient (see Theorem~\ref{tm1},
Proposition~\ref{prop1n11} and Corollary~\ref{cor1n11}) and necessary
conditions (see Theorem~\ref{tm2}) for a SP to be canonical. In
Section~\ref{secnoncanon} we consider non-canonical SPs with two sign variations
and we give a lower bound on the number of different orders of the moduli
of roots for which these SPs are realizable by~HPs.

\section{Preliminaries}

\begin{nota}
  {\rm (1) We set $\sigma ^m(P)=\sigma ((-1)^dP(-x))$ and
    $\sigma ^r(P)=\sigma (x^dP(1/x)/P(0))$.
    \vspace{1mm}
    
    (2) We call {\em first representation} of a SP the one with signs $(+)$ and/or $(-)$. For a SP in its {\em second representation} $\Sigma _{m_1,\ldots ,m_s}$, if each of its
    maximal sequences of, say, $k$ consecutive units is replaced by the symbol 
    $[k]$, then one obtains the {\em third representation} of the SP. E.g. the SP

    $$(+,-,-,+,-,+,-,-,-)~=~\Sigma _{1,2,1,1,1,3}$$
    can be represented also in the form $\Sigma _{[1],2,[3],3}$. We call the signs $(+)$ and $(-)$ of the first
    representation and the numbers $m_i$ of the second one {\em components}
    of the SP. The
    components larger than $1$ and the maximal sequences of units in the third
    representation are called
    {\em elements} of the~SP.}
\end{nota}

\begin{rems}\label{rems2}
  {\rm (1) The polynomial $x^dP(1/x)$ is the {\em reverted} of the polynomial
    $P$ (i.e. read from the back). Its roots are the reciprocals of the roots
    of $P$. The roots of $P(-x)$ are the opposite of the roots of $P$.
    \vspace{1mm}
    
    (2) The applications

    $$\iota_m:\sigma (P)\mapsto \sigma ^m(P)~~~\, \, \, {\rm and}~~~\, \, \,   
    \iota_r:\sigma (P)\mapsto \sigma ^r(P)$$
    are two commuting involutions. We set

    $$\sigma ^{mr}(P)~:=~
    \sigma ^m(\sigma ^r(P))~=~\sigma ^r(\sigma ^m(P))~=~:\sigma ^{rm}(P))~.$$
    For $d\geq 1$, it is always true that
    $\sigma (P)\neq \sigma ^m(P)$ (because their second signs are opposite),
    but one might have $\sigma (P)=\sigma ^r(P)$ or
    $\sigma (P)=\sigma ^{mr}(P)$. Thus the set $\{ \sigma (P)$, $\sigma ^m(P)$,
    $\sigma ^r(P)$, $\sigma ^{mr}(P)\}$
    contains either four or two distinct~SPs.
    \vspace{1mm}
    
    (3) The SPs $\sigma (P)$, $\sigma ^m(P)$, $\sigma ^r(P)$ and
    $\sigma ^{mr}(P)$ are simultaneously canonical or not. With regard to
    Example~\ref{ex1} one has}

  $$\begin{array}{ll}
    \sigma ^r(\Sigma _{1,m_2})=\Sigma _{m_2,1}~,&\\ 
  \sigma ^m(\Sigma _{1,m_2})=\Sigma _{2,[m_2-1]}~,&
  \sigma ^{mr}(\Sigma _{1,m_2})=\Sigma _{[m_2-1],2}~,\\ \\
  \sigma ^r(\Sigma _{m_1,1,m_3})=\Sigma _{m_3,1,m_1}~,&\\ 
  \sigma ^m(\Sigma _{m_1,1,m_3})=\Sigma _{[m_1-1],3,[m_2-1]}~,&
  \sigma ^{mr}(\Sigma _{m_1,1,m_3})=\Sigma _{[m_2-1],3,[m_1-1]}~,\\ \\ 
  \sigma ^r(\Sigma _{1,m_2,1})=\Sigma _{1,m_2,1}~~~{\rm and}& 
  \sigma ^m(\Sigma _{1,m_2,1})=\sigma ^{mr}(\Sigma _{1,m_2,1})=
  \Sigma _{2,[m_2-2],2}~.
  \end{array}$$
  \end{rems}

\begin{defi}\label{defitype1}
  {\rm (1) We say that a SP $\sigma$ of length $d+1$ is {\em of type 1}
    (notation: $\sigma \in \mathcal{T}_{1,d}$) if either
    all its even or all its odd positions contain the same sign.
    \vspace{1mm}
    
    (2)  We say that a SP $\sigma$ of length $d+1$ is {\em of type 2}
    (notation: $\sigma \in \mathcal{T}_{2,d}$) if
    \vspace{1mm}
    
    (i) in its second representation the SP $\sigma$ does not have two consecutive components $m_i$
    larger than $1$;
    \vspace{1mm}
    
    (ii) for $2\leq i\leq s-1$, one has $m_i\neq 2$ (but $m_1=2$ and/or $m_s=2$
    is allowed).}
\end{defi}

\begin{rem}
  {\rm SPs of type 1 are used in the formulation of a result concerning another
problem connected with Descartes' rule of signs and formulated for real
(not necessarily hyperbolic) polynomials, see Proposition~4 in~\cite{FoKoSh}.}
  \end{rem}

\begin{ex}\label{ex2}
    {\rm (1) For $d=6$ and for the
      SP

      $$\sigma ^{\dagger}~:=~(+,-,-,-,+,-,+)~=~\Sigma _{1,3,1,1,1}~=~
      \Sigma _{[1],3,[3]}$$
      one has $\mathcal{T}_{1,6}\ni \sigma ^{\dagger}\in \mathcal{T}_{2,6}$,
    because there are $(-)$-signs in all odd positions  
    (namely, $1$, $3$ and $5$) and conditions (i) and (ii) from
    Definition~\ref{defitype1} hold true.
    \vspace{1mm}
    
    (2) The
    SP $\sigma _0$ from Example~\ref{defiCPP} is
    neither a type 1 nor a type 2 SP.
    \vspace{1mm}
    
    (3) One has
    $\mathcal{T}_{1,7}\not\ni \Sigma_{[1],4,[3]}\in \mathcal{T}_{2,7}$.
    \vspace{1mm}
    
    (4) The following SPs are of type 1: $\Sigma _{A,[B]}$, $\Sigma _{[A],B}$,
    $\Sigma _{A,[2B+1],C}$, $\Sigma _{[A],2B+1,[C]}$, $A$, $B$, $C\in \mathbb{N}$.}
\end{ex}

\begin{rem}
  {\rm The SPs $\sigma (P)$, $\sigma ^m(P)$, $\sigma ^r(P)$ and
    $\sigma ^{mr}(P)$ are simultaneously of type 1 or not. E.g. of type 1 are
    the SPs $\sigma _{\bullet}:=\Sigma _{m_1,[u],m_s}$ for $s$ odd (with
    $u:=d+1-m_1-m_s$),}

  $$\sigma _{\bullet}^m=\Sigma _{[m_1-1],u+2,[m_s-1]}~,~~~\, 
  \sigma _{\bullet}^r=\Sigma _{m_s,[u],m_1}~~~\, {\rm and}~~~\,
  \sigma _{\bullet}^{mr}=\Sigma _{[m_s-1],u+2,[m_1-1]}~.$$
\end{rem}

\begin{prop}
  (1) One has $\mathcal{T}_{1,d}\subset \mathcal{T}_{2,d}$.
  \vspace{1mm}

  (2) One has $\iota_m( \mathcal{T}_{2,d})=\mathcal{T}_{2,d}$ and
  $\iota_r( \mathcal{T}_{2,d})=\mathcal{T}_{2,d}$.
\end{prop}

\begin{proof}
  Part (1). Indeed, if condition (i) of Definition~\ref{defitype1}
  does not hold true, then the SP is of the form

  $$(\cdots ,+,+,-,-,\cdots )~~~\, \, \, {\rm or}~~~\, \, \, 
  (\cdots ,-,-,+,+,\cdots )$$
  and each of the sequences of signs of even or odd monomials has
  a sign variation.
  If condition (ii) does not hold true, then the SP is of
  the form 
$$(\cdots ,-,+,+,-,\cdots )~~~\, \, \, {\rm or}~~~\, \, \, 
  (\cdots ,+,-,-,+,\cdots )$$
  and again each of these sequences has at least one sign variation.
  \vspace{1mm}
  
  Part (2). The inclusion $\iota_r( \mathcal{T}_{2,d})\subset \mathcal{T}_{2,d}$
  follows directly from Definition~\ref{defitype1}. As $\iota_r$ is an
  involution, this inclusion is an equality. For a SP
  $\sigma ^{\triangle}\in \mathcal{T}_{2,d}$, its image
  $\iota _m(\sigma ^{\triangle})$ is defined by the following rules:
  \vspace{1mm}
  
  (a) An element $A>1$ of $\sigma ^{\triangle}$ is replaced by $[A-2]$
  if $A$ is not at one of the ends of $\sigma ^{\triangle}$,
  and by $[A-1]$ if it is.
  \vspace{1mm}

  (b) An element $[B]$ of $\sigma ^{\triangle}$ is replaced by $B+2$
  if $[B]$ is not at one of the ends of $\sigma ^{\triangle}$,
  and by $B+1$ if it is.
  \vspace{1mm}

  One can deduce from rules (a) and (b) that conditions (i) and (ii) hold true
  for the SP~$\iota _m(\sigma ^{\triangle})$. Hence
  $\iota_r( \mathcal{T}_{2,d})\subset \mathcal{T}_{2,d}$ and as $\iota_m$ is an
  involution, this inclusion is an equality.

\end{proof}

\section{Results on canonical sign patterns}

\begin{tm}\label{tm1}
  Every type 1 SP is canonical.
\end{tm}

\begin{proof}
  We prove the theorem by induction on
  $d$. For $d=1$, there is nothing to prove.
  For $d=2$, one has to consider the SPs $\sigma ^{\sharp}:=(+,+,-)$ and
  $\sigma ^{\flat}:=(+,-,-)$. For a HP $P:=(x-a)(x+b)=x^2+g_1x+g_0$, one has
  $g_1=b-a$ which is $>0$ if $b>a$ and $<0$ for $b<a$ from which for $d=2$ the
  theorem follows.

  Suppose that $d\geq 3$.
  Consider a HP $P$ with all roots
  simple defining a SP
  $\sigma$. In the one-parameter family of polynomials $P^*_t:=tP+(1-t)P'$,
  $t\in [0,1]$, every polynomial is hyperbolic with all roots simple
  and for $t\in (0,1]$, all roots
  of $P^*_t$ are nonzero. Moreover, for $t\in (0,1]$, the polynomial $P^*_t$
  defines the SP $\sigma$. 

      For $t=0$, by inductive assumption, the moduli of the roots of the HP $P'$
      define the canonical order. For $t\in (0,1]$, there is no equality
        between a modulus of a positive and a modulus of a negative root of
        $P^*_t$. Indeed, if $P^*_t$ has roots $\pm \gamma$, $\gamma >0$, then  

        \begin{equation}\label{eqQpm}
          Q_{t,\pm}(\gamma )~:=~P^*_t(\gamma )~\pm
          ~P^*_t(-\gamma )~=~0~.\end{equation}
        This is impossible,
        because at least one of the two quantities $Q_{t,\pm}(\gamma )$ is a sum of
        terms of the same sign. 

        Thus the $d-1$ largest of the moduli of roots of $P^*_t$ define the
        same order as the roots of $P^*_0$ (which is the canonical
        order w.r.t. the SP
        obtained from $\sigma$ by deleting its last component). The root of
        least modulus (for $t$ close to $0$) is positive if the last two
        components of $\sigma$ are different and negative if they are equal.
        Thus for $t\in (0,1]$, the moduli of the roots of $P^*_t$ (hence in
        particular the ones of $P^*_1$)  define the canonical order.

\end{proof}

\begin{tm}\label{tm2}
A canonical SP is a type 2 SP.
%
\end{tm}

\begin{rem}
{\rm Theorem~\ref{tm2} proposes necessary conditions for a SP to be canonical.
  It would be interesting to know how far from sufficient these
  conditions are.}
\end{rem}

\begin{proof}
  Suppose that a given SP $\sigma$ has components $m_j=A>1$ and $m_{j+1}=B>1$.
  The SP $\Sigma _{A,B}$ is not
  canonical, see part (2) of Example~\ref{ex1}. Hence one can construct
  two polynomials $P$ and $Q$ defining the SP $\Sigma _{A,B}$ and with
  different orders of their moduli of roots. To construct two polynomials
  realizing the SP $\sigma$ one starts with $P$ and $Q$ and then uses
  concatenation of polynomials as described in the proof of
  Proposition~\ref{propevery} and in Remarks~\ref{rems1}. At each new
  concatenation the modulus of the new root is either much smaller or much
  larger than the moduli of the previously existing roots. Hence the orders
  of the moduli of the roots of the two polynomials constructed in this way
  after $P$ and $Q$ remain different. 

  If the SP $\sigma$ has a component $m_i=2$, $2\leq i\leq s-1$, then it suffices to consider the case
  $m_{i-1}=m_{i+1}=1$. In this case one chooses two polynomials $P$ and $Q$
  realizing the SP $\Sigma _{1,2,1}$ with different orders of the moduli of
  their roots; such polynomials exist, see part (1) of Example~\ref{ex1}.
  After this one again uses the techniques of concatenation to realize the
  SP $\sigma$ with two 
  different orders of the moduli of the roots,
  starting with $P$ and $Q$ respectively.
  \end{proof}

\begin{prop}\label{prop1n11}
  For $d\geq 5$, the SP $\Sigma _{[1],d-2,[2]}$ is canonical.
\end{prop}

\begin{cor}\label{cor1n11}
  For $d\geq 5$, the three SPs
  $\Sigma_{[2],d-2,[1]}=\sigma ^r(\Sigma _{[1],d-2,[2]})$,
  $\Sigma_{2,[d-4],3}=\sigma ^m(\Sigma _{[1],d-2,[2]})$ and
  $\Sigma_{3,[d-4],2}=\sigma ^{mr}(\Sigma _{[1],d-2,[2]})$ are canonical.
\end{cor}

The corollary follows from part (3) of Remarks~\ref{rems2}.

\begin{proof}[Proof of Proposition~\ref{prop1n11}]
  For $d\geq 5$ odd, the SP is of type 1 and one can apply Theorem~\ref{tm1}.
  For $d=4$, the SP is not canonical,
  see Theorem~\ref{tm2}. So we assume that $d\geq 6$
  (the parity of $d$ is of no importance in the proof). 
  Without loss of generality one can assume that the middle modulus of a
  positive root of a HP $P:=x^d+\sum _{j=0}^{d-1}a_jx^j$ realizing the SP $\Sigma _{[1],d-2,[2]}$ equals $1$ (this can be achieved by a linear change of the
  variable $x$). So we denote the moduli of positive roots by
  $0<\varepsilon <1<A$, and by $0<\gamma _1<\gamma _2<\cdots <\gamma _{d-3}$
  the moduli of negative roots. 
  
  Denote by
  $0<\delta _1<\cdots <\delta _{d-3}$ the moduli of negative and by
  $0<\varphi <\psi$ the moduli of positive roots of $P'$ (recall that $P'$
  defines the SP $\Sigma _{1,d-2,1}$ which is canonical, see part (1) of
  Example~\ref{ex1}). As $\Sigma _{1,d-2,1}$ is canonical, one has 
  $\varphi <\delta _1$, and by Rolle's theorem,
  $\gamma _j<\delta _{j+1}<\gamma _{j+1}$, $j=1$, $\ldots$, $d-4$.
  For $\delta _1$, one has $0<\delta _1<\gamma _1$.
  Thus
  $$\varepsilon ~<~\varphi ~<~\delta _1~<~\gamma _1~.$$

  Denote by $0<\eta _1<\cdots <\eta _{d-3}$ the moduli of negative and by
  $0<\lambda <\mu$ the moduli of positive roots of the HP
  $P^{\dagger}:=xP'-dP=-a_{d-1}x^{d-1}-2a_{d-2}x^{d-2}-\cdots -da_0$.
  The
  latter defines the SP $\Sigma _{d-2,[2]}$ which is canonical, see part (1) of
  Example~\ref{ex1}. The positive roots of $P$ and $P^{\dagger}$ interlace, and so do their negative roots as well; we will see below that this is not true about all the roots of $P$ and $P^{\dagger}$. The
  leading coefficient of $P^{\dagger}$ is positive, so the limits at $+\infty$
  of $P$ and $P^{\dagger}$ equal $+\infty$. Their limits at $-\infty$
  are opposite. The leftmost root of $P$ equals $-\gamma _{d-3}$. One has
  $P^{\dagger}(-\gamma _{d-3})=-\gamma _{d-3}P'(-\gamma _{d-3})$. Hence 

  $$\begin{array}{lll}
    {\rm either}&\lim _{x\rightarrow -\infty}P(x)~=~-\infty ~,&
    P'(-\gamma _{d-3})~>~0~,\\ \\ 
    &\lim _{x\rightarrow -\infty}P^{\dagger}(x)~=~+\infty ~,&
    P^{\dagger}(-\gamma _{d-3})~<~0
         \\ \\ {\rm or}&
    \lim _{x\rightarrow -\infty}P(x)~=~+\infty ~,&P'(-\gamma _{d-3})~<~0~,\\ \\ 
&\lim _{x\rightarrow -\infty}P^{\dagger}(x)~=~-\infty ~,&
P^{\dagger}(-\gamma _{d-3})~>~0~.
  \end{array}$$
  In both cases the leftmost root $-\eta _{d-3}$
  of $P^{\dagger}$ is $<-\gamma _{d-3}$. By Rolle's theorem and using the fact
  that the SP $\Sigma _{d-2,[2]}$ is canonical,

  $$0~<~\lambda ~<~1~<~\mu ~<~\eta _1~<~\gamma _2~<~\eta _2~.$$
One can show that $-\eta _1<-\gamma _1<\varepsilon <\lambda$ which means that the interlacing of the roots of $P$ and $P^{\dagger}$ is interrupted when the variable $x$ passes through~$0$. The condition $a_{d-1}<0$ reads:

\begin{equation}\label{eqAeps}A+1+\varepsilon -\sum _{j=1}^{d-3}\gamma _j~>~0~.
\end{equation}
As $\gamma _1>\varepsilon$ and $\gamma _2>1>\varepsilon$,
condition (\ref{eqAeps}) is
possible only if $A>\gamma _{d-3}$. Thus
$\varepsilon <\gamma _1<\cdots <\gamma _{d-3}<A$ and to prove the proposition
there remains to show that $1<\gamma _1$. 
Set 

$$\begin{array}{cclccclc}\sigma _1&:=&\sum _{j=3}^{d-3}1/\gamma _j&,&
  \sigma _2&:=&\sum _{3\leq i<j\leq d-3}1/(\gamma _i\gamma _j)&,\\ \\ B&:=&
  \frac{1}{A}+1+\frac{1}{\varepsilon}&{\rm and}&C&:=&\frac{1}{A\varepsilon}+
  \frac{1}{A}+\frac{1}{\varepsilon}&.\end{array}$$ 
The conditions $a_0<0$, $a_1>0$ and $a_2<0$ imply 

\begin{equation}\label{eqtwoineq}\begin{array}{lclclclc}
    &&B&-&\left( \frac{1}{\gamma _1}+\frac{1}{\gamma _2}+\sigma _1\right) &>&0&
    {\rm and}\\ \\ 
    \Phi&:=&\Lambda&-&B \left( \frac{1}{\gamma _1}+\frac{1}{\gamma _2}+
    \sigma _1\right) &>&0~,&{\rm where}\\ \\ \Lambda&:=&C&
    +&\frac{1}{\gamma _1\gamma _2}+\left( \frac{1}{\gamma _1}+
    \frac{1}{\gamma _2}\right) \sigma _1+\sigma _2&.&&
                                 \end{array}
\end{equation}
Suppose that $\gamma _1\leq 1$. Then the following inequalities hold true:

\begin{equation}\label{eqineq1}
 \frac{1}{A\varepsilon}-\frac{1}{\gamma _2\varepsilon}~<~0~,
\end{equation}
because $A>\gamma _2$,
\begin{equation}\label{eqineq2}
  -\frac{1}{\gamma _1}\left( \frac{1}{A}+\frac{1}{\varepsilon}\right) +
  \frac{1}{A}+\frac{1}{\varepsilon}~\leq ~0~,
\end{equation}

\begin{equation}\label{eqineq3}
  -B\sigma _1+\left( \frac{1}{\gamma _1}+\frac{1}{\gamma _2}\right) \sigma _1+
  \sigma _2~<~0,
\end{equation}
(because $-B\sigma _1<-(1/\gamma _1+1/\gamma _2+\sigma _1)\sigma _1$,
see the first of
inequalities (\ref{eqtwoineq}), and one has $\sigma _2<(\sigma _1)^2$) and as
$\gamma _2>1$, 

\begin{equation}\label{eqineq4}
 -\frac{1}{\gamma _1}+\frac{1}{\gamma _1\gamma _2}~<~0.
\end{equation}
The sum of the left-hand sides of inequalities (\ref{eqineq1}),
(\ref{eqineq2}), (\ref{eqineq3}) and 
(\ref{eqineq4}) equals

$$
  \Lambda -\left( \frac{1}{\gamma _2\varepsilon}+\frac{1}{\gamma _1}\left(
  \frac{1}{A}+\frac{1}{\varepsilon}\right) +B\sigma _1+\frac{1}{\gamma _1}
  \right) =\Lambda -B\left( \frac{1}{\gamma _1}+\frac{1}{\gamma _2}+
  \sigma _1\right) +\frac{1}{\gamma _2}+\frac{1}{\gamma _2A}~.$$
Thus $\Lambda -B\left( \frac{1}{\gamma _1}+\frac{1}{\gamma _2}+
\sigma _1\right) +\frac{1}{\gamma _2}+\frac{1}{\gamma _2A}<0$ which
  contradicts the second of inequalities
  (\ref{eqtwoineq}). \end{proof}

\section{On non-canonical sign patterns\protect\label{secnoncanon}}

The present section deals with SPs with two sign changes, i.e. with $s=3$,
see Notation~\ref{nota1}.
For $m_1\geq 2$, $m_2\geq 2$, $m_3\geq 2$, such a SP is not canonical, see
Theorem~\ref{tm2}.

\begin{nota}
{\rm We set $m:=m_1$, $n:=m_2$, $q:=m_3$ and we denote by $0<\beta <\alpha$
the positive and by $-\gamma _{d-2}<\cdots <-\gamma _1<0$ the negative roots of
a degree $d$ HP $P$ realizing the SP $\Sigma _{m,n,q}$. By $m^*$, $n^*$, $q^*$
we denote the numbers of negative roots of modulus larger than $\alpha$,
between $\beta$ and $\alpha$ and smaller than $\beta$ respectively; hence
$m^*+n^*+q^*=d-2$. 
By $\tau _1\geq 0$, $\tau _2\geq 0$, $\delta>0$, $\ell >0$ and $r\geq 2$,
we denote integers, where $d=\delta +\tau _1+\tau _2$.}
\end{nota}

We remind that the canonical order of the roots corresponds to the case
$m^*=m-1$, $n^*=n-1$, $q^*=q-1$, see Definition~\ref{defiCO}. 

\begin{tm}\label{tm3}
  (1) For

  $$\begin{array}{lll}r^2~<~\delta ~<~(r+1)^2~,&\delta -r~\in ~2\mathbb{Z}+1~,&
    \\ \\ 
    m~\geq ~(\delta -r+1)/2~,&q~\geq ~(\delta -r+1)/2&{\rm and}~~~\, \, \, 
    n~=~r~,\end{array}$$
  the SP
  $\Sigma _{m,n,q}$ is
  realizable by HPs with all possible values of $m^*$, $n^*$, $q^*$ such that
  $m^*\geq \tau _1:=m-(\delta -r+1)/2$ and $q^*\geq \tau _2:=q-(\delta -r+1)/2$.
  \vspace{1mm}
  
  (2) For

  $$\begin{array}{lll}r^2~<~\delta ~<~(r+1)^2~,&\delta -r~\in ~2\mathbb{Z}~,&
    \\ \\ 
    m~\geq ~(\delta -r)/2~,&q~\geq ~(\delta -r)/2&{\rm and}~~~\, \, \, n~=~r+1~,
  \end{array}$$
  the SP
  $\Sigma _{m,n,q}$ is
  realizable by HPs with all possible values of $m^*$, $n^*$, $q^*$ such that
  $m^*\geq \tau _1:=m-(\delta -r)/2$ and $q^*\geq \tau _2:=q-(\delta -r)/2$.
  \vspace{1mm}
  
  (3) For $\delta =r^2$, the SP

  $$(~\tau _1+r(r-1)/2+1~,~r~,~\tau _2+r(r-1)/2~)~~~({\rm resp.}~~~
  (~\tau _1+r(r-1)/2~,~r~,~\tau _2+r(r-1)/2+1~))$$
  is realizable by HPs with all possible values of $m^*$, $n^*$, $q^*$
  such that $m^*\geq \tau _1+1$ and $q^*\geq \tau _2$ (resp.
  $m^*\geq \tau _1$ and $q^*\geq \tau _2+1$).
  \end{tm}

\begin{rems}
  {\rm (1) Consider the case $\tau _1=\tau _2=0$. Hence $d=\delta$ and
    all possible orders of the moduli of the $d-2$ negative and $2$
    positive roots are realizable. The number of these orders is}

  $$\sum _{k=0}^{d-2}~\sum _{j=0}^{d-2-k}1~=~\sum _{k=0}^{d-2}(d-1-k)~=~d(d-1)/2$$
  {\rm (here $k$ and $j$ are the numbers of moduli of negative roots
    larger than $\alpha$ and between $\beta$ and $\alpha$ respectively).
    At the same time $d\sim r^2$, i.e. $d\sim n^2$.
    Thus the theorem guarantees the
    possibility to realize the SP $\Sigma _{m,n,q}$ by $\sim n^4/2$ HPs with 
    different orders of the moduli of their roots when $m$ and $q$ are (almost)
    equal. The latter condition is essential -- for $q=1$, the number of
    different orders is $\sim 2n$, see Theorem~4 in~\cite{Koarxiv}.
    \vspace{1mm}
    
    (2) The theorem gives only sufficient conditions for realizability of
    certain SPs with two sign changes by HPs with different orders of the
    moduli of their roots. It would be interesting to obtain necessary
    conditions as well.}
  \end{rems}

In order to prove the theorem we need a technical lemma.

\begin{nota}
  {\rm We set $P_{\ell}(x):=(x-1)^2(x+1)^{\ell}$, $\ell \geq 2$. This polynomial
    contains either $0$ or $2$ vanishing coefficients, see Lemma~\ref{lm1}. 
    By $\Sigma (\ell)$ we
    denote its SP which, in the case when there are $2$ vanishing coefficients,
    we represent in the form $(v,0,n,0,w)$. This means
    that $\Sigma (\ell)$ begins with $v=m-1$ signs~$(+)$ followed by a zero
    followed by $n=n(\Sigma (\ell ))$ signs~$(-)$ followed by a zero followed
    by $w=q-1$ signs~$(+)$. If there are no vanishing coefficients, then we
  write $\Sigma (\ell)=(v,n,w)$ in which case $v=m$ and $w=q$.}
\end{nota}

\begin{lm}\label{lm1}  
  (1) For~~$r^2-2~<~\ell ~<~(r+1)^2-2$~~and~~$\ell -r~\in ~2\mathbb{Z}+1$,
  
  $$\Sigma (\ell )~=~(~(\ell -r+3)/2~,~r~,~(\ell -r+3)/2)~)~,$$
  so~~$n(\Sigma (\ell ))=r$.
  \vspace{1mm}
  
  (2) For~~$r^2-2~<~\ell ~<~(r+1)^2-2$~~and~~$\ell -r~\in ~2\mathbb{Z}$, 

  $$\Sigma (\ell )~=~(~(\ell -r+2)/2~,~r+1~,~(\ell -r+2)/2~)~,$$
  so~~$n(\Sigma (\ell ))~=~r+1$.
  \vspace{1mm}
  
  (3) For~~$\ell ~=~r^2-2$, the~~SP~~$\Sigma (\ell )$~~equals

  $$\Sigma (\ell )~=~(~r(r-1)/2~,~0~,~r-1~,~0~,~r(r-1)/2~)~.$$
  Hence~~$n(\Sigma (\ell ))~=~r-1$.
  
\end{lm}

\begin{proof}
  Clearly $P_{\ell}=\sum _{j=0}^{\ell +2}c_jx^j$, where
  $c_j={\ell \choose j}-2{\ell \choose j-1}+{\ell \choose j-2}$. The condition
  $c_j=0$ is equivalent to

  $$4j^2-4(\ell +2)j+(\ell +1)(\ell +2)~=~0$$
  which yields
  $$j~=~j_{\pm}(\ell )~:=~(\ell +2\pm \sqrt{\ell +2})/2~.$$
  For $\ell =r^2-2$, one gets $j=(r^2\pm r)/2$ from which part (3) follows
  (both numbers $(r^2\pm r)/2$ are natural).

  When $\ell$ is not of the form $r^2-2$ the condition $c_j=0$ does not provide
  a natural solution. Hence no coefficient of $P_{\ell}$ vanishes. The formula
  expressing $j_{\pm}(\ell )$ implies that $c_j>0$ for
  $j\leq [(\ell +2-(r+1))/2]$ while $c_{[(\ell +2-(r+1))/2]+1}<0$; here $[.]$
  stands for the integer part of. If $\ell$ and $r$
  are of different parity, then

  $$[(\ell +2-(r+1))/2]~=~(\ell -r+1)/2$$
  which proves part (1). If $\ell$ and $r$ are of one and the same parity,
  then
  $$[(\ell +2-(r+1))/2]~=~(\ell -r)/2$$
  which proves part (2).

\end{proof}

\begin{proof}[Proof of Theorem~\ref{tm3}]
  To prove parts (1) and (2) of Theorem~\ref{tm3} we use parts (1) and (2) of
  Lemma~\ref{lm1} respectively. We consider first the case
  $\tau _1=\tau _2=0$. In this
  case the conditions

  $$\begin{array}{lllll}
    m~\geq ~(\delta -r+1)/2~,&q~\geq ~(\delta -r+1)/2&{\rm and}&n~=~r&
  {\rm from~part~(2)~or}\\ \\ m~\geq ~(\delta -r)/2~,&
  q~\geq ~(\delta -r)/2&{\rm and}&n~=~r+1&{\rm from~part~(3)}\end{array}$$
  are possible only if

$$\begin{array}{lllll}
    m~=~(\delta -r+1)/2~,&q~=~(\delta -r+1)/2&{\rm and}&n~=~r&{\rm or}\\ \\
    m~=~(\delta -r)/2~,&q~=~(\delta -r)/2&{\rm and}&n~=~r+1&\end{array}$$
  respectively, because $m+n+q=\delta +1$.

  Set $d=\delta :=\ell +2$. We deform the polynomial
  $P_{\ell}$ corresponding to part (1) or (2) of Lemma~\ref{lm1} so that the moduli of the roots
  are all distinct and define any possible order (fixed in advance)
  on the positive half-axis. The positive
  roots $\beta <\alpha$ of the deformed polynomial
  (denoted by $\tilde{P}_{\ell}$)
  remain close to $1$ and
  the $\ell$ negative roots remain close to $-1$. Hence the signs of
  the coefficients of $\tilde{P}_{\ell}$ are the same as the signs of the
  coefficients of $P_{\ell}$ and

  $$\begin{array}{ll}
    \sigma (\tilde{P}_{\ell})~=~(~(\delta -r+1)/2~,~r~,~(\delta -r+1)/2~)&
           {\rm in~the~case~of~part~(2)~or}\\ \\
           \sigma (\tilde{P}_{\ell})~=~(~(\delta -r)/2~,~r+1~,~(\delta -r)/2~)&
                  {\rm in~the~case~of~part~(3).}
  \end{array}$$
  This proves the theorem for $\tau _1=\tau _2=0$.

  In the general case, i.e.
  for $\tau _1\geq 0$ and $\tau _2\geq 0$, one first constructs the polynomial
  $\tilde{P}_{\ell}$ as above. Then one performs $\tau _1$ concatenations of
  $\tilde{P}_{\ell}$ with polynomials of the form $1+\varepsilon _jx$, $j=1$,
  $\ldots$, $\tau _1$, as
  explained in part (2) of Remarks~\ref{rems1}, where

  $$0~<~\varepsilon _{\tau _1}~\ll ~\varepsilon _{\tau _1-1}~\ll ~\cdots ~\ll ~
  \varepsilon _1~\ll ~1~.$$
  This adds $\tau _1$ negative roots whose moduli are larger than~$\alpha$.
  After this one performs $\tau _2$ concatenations, see part (1) of
  Remarks~\ref{rems1}, with polynomials of the form $x+\varepsilon _j$,
  $j=\tau _1+1$, $\ldots$, $\tau _1+\tau _2$, where

  $$0~<~\varepsilon _{\tau _1+\tau _2}~\ll ~\varepsilon _{\tau _1+\tau _2-1}~\ll ~
  \cdots ~\ll ~\varepsilon _{\tau _1+1}~\ll ~\varepsilon _{\tau _1}~.$$
  This adds $\tau _2$ negative roots whose moduli are smaller than $\beta$.

  Part (3). Consider first the case $\tau _1=\tau _2=0$.
  We use Lemma~\ref{lm1} with $\ell =r^2-3$. Hence one can apply
  part (2) of Lemma~\ref{lm1} with $r-1$ substituted for $r$
  (because $\ell -(r-1)\in 2\mathbb{Z}$). This implies that
  the polynomial $P_{r^2-3}$ realizes the SP $(r(r-1)/2,r,r(r-1)/2)$. Setting
  $P_{r^2-3}:=\sum _{j=0}^da_jx^j$ one deduces that

  $$\begin{array}{llllll}
a_{r(r-1)/2}~<~0&,&a_{r(r-1)/2-1}~>~0&,&a_{r(r-1)/2}+a_{r(r-1)/2-1}~=~0&,\\ \\ 
a_{r(r+1)/2+1}~>~0&,&a_{r(r+1)/2}~<~0&,&a_{r(r+1)/2+1}+a_{r(r+1)/2}~=~0&.\end{array}$$
  The two equalities to $0$ result from the polynomial
  $P_{r^2-2}=(x+1)P_{r^2-3}$ having vanishing coefficients of $x^{r(r-1)/2}$ and
  $x^{r(r+1)/2}$, see part (3) of Lemma~\ref{lm1}. Hence for the SPs defined by
  the polynomials $P_{\pm}:=(x+1\pm \varepsilon )P_{r^2-3}$, $\varepsilon >0$,
  one has

  $$\sigma (P_+)~=~(r(r-1)/2+1,r,r(r-1)/2)~~~\, \, {\rm and}~~~\, \, 
  \sigma (P_-)~=~(r(r-1)/2,r,r(r-1)/2+1)~.$$
  Then one perturbs the roots of $P_{r^2-3}$ (the perturbed negative roots must
  keep away from the root $-1\mp \varepsilon$ of $P_{\pm}$).
  In the case of $P_+$ (resp. $P_-$) the largest (resp. the smallest)
  of the moduli of perturbed roots is the one of the negative
  root $-1-\varepsilon$ (resp. $-1+\varepsilon$) and the order of the remaining
  $d-3$ negative and $2$ positive roots can be arbitrary. This proves part (3)
  for $\tau _1=\tau _2=0$. In the general case, i.e.
  for $\tau _1\geq 0$ and $\tau _2\geq 0$, the proof is finished in the same
  way as for parts (1) and~(2).

  \end{proof}

\end{document}